\documentclass[12pt]{amsart}
\usepackage[left=1in,right=1in,top=1.2in,bottom=1in]{geometry}
\usepackage{amsfonts}
\usepackage{amsmath}
\usepackage{amssymb}
\usepackage{amscd}
\usepackage{mathtools}
\usepackage{color}
\newcommand{\Mod}[1]{\ (\mathrm{mod}\ #1)}
\allowdisplaybreaks[1]
\newtheorem{theorem}{Theorem}[section]
\newtheorem{lemma}[theorem]{Lemma}

\newtheorem{corollary}[theorem]{Corollary}
\theoremstyle{definition}
\newtheorem{definition}[theorem]{Definition}

\theoremstyle{remark}

\numberwithin{equation}{section}
\usepackage{hyperref}

\hypersetup{
	bookmarks=true,         		pdftitle={},        pdfauthor={Chiranjit Ray},         colorlinks=true,
	linkcolor={blue}, urlcolor={blue},
	citecolor={red}, anchorcolor = {blue}}
\begin{document}
\title[$MEX$ related  integer partitions of Andrews and Newman]{Divisibility and distribution of  $MEX$ related  integer partitions of Andrews and Newman}
\author{Chiranjit Ray}
\address{Mathematics and Data Science Group, Indian Institute of Information Technology Sri
	City, Tirupati District - 517 646, Andhra Pradesh, India}
\curraddr{}
\email{chiranjitray.m@iiits.in,~ chiranjitray.m@gmail.com}

\date{September 10, 2021}
\thanks{} 
\begin{abstract}
Andrews and Newman introduced the minimal excludant or  ``$mex$'' function for an integer partition $\pi$ of a positive integer $n$, $mex(\pi)$, as the smallest positive integer that is not a part of $\pi$. They defined $\sigma mex(n)$ to be the sum of $mex(\pi)$ taken over all partitions $\pi$ of $n$. We  prove  infinite families of congruence and multiplicative formulas for $\sigma mex(n)$. By restricting to the part of $\pi$,  Andrews and Newman also introduced $moex(\pi)$ to be the smallest odd integer that is not a part of $\pi$ and $\sigma moex(n)$ to be the sum of $moex(\pi)$ taken over all partitions $\pi$ of $n$. In this article, we show that for any sufficiently large $X$, the number of all positive integer $n\leq X$ such that $\sigma moex(n)$ is an even (or odd) number is at least $\mathcal{O}(\log \log X)$.
\end{abstract}
\maketitle
\section{Introduction and statement of results}
Fraenkel and Peled \cite{Fraenkel2015} defined the minimal excludant or ``$mex$'' function on a set $S$ of positive integers is the least positive integer not in $S$. Perhaps the notion of the $mex$ function was introduced in  1930's  and is known for its uses in combinatorial game theory \cite{Grundy1939, Sprague1935}.

\par A partition of a non-negative integer $n$ is a non-increasing sequence of positive integers whose sum is $n$. Let $\pi$ be a partition of a positive integer $n$,  and let the set of all partitions of $n$ be denoted by $\mathcal{P}(n)$. Recently, Andrews and Newman \cite{Andrews2019} considered the {\it minimal excludant function}  applied to integer partitions. The minimal excludant of $\pi$, $mex(\pi)$, is the smallest positive integer which is not a part of $\pi$.  Thus if $\pi$ is $6+4+3+2+1$, then $mex(\pi)=5$.   Then, for each positive integer $n$, they defined $$\sigma mex(n)~:=\sum_{\pi \in \mathcal{P}(n)} mex (\pi).$$
For example,  $\sigma mex(4)=9$ with the relevant $mex$ partitions being:  $mex(4)=1$, $mex(3+1)=2$, $mex(2+2)=1$, $mex (2+1+1)=3$, and  $mex(1+1+1+1)=2$. The generating function of $\sigma mex(n)$ is 
\begin{align}\label{1.1}
\sum_{n=0}^{\infty}\sigma mex(n) q^n = (-q;q)_{\infty}^2,
\end{align}
where  the $q$-shifted factorial $(a; q)_{\infty}:=\prod_{n=1}^{\infty}(1-aq^{n-1}), ~~|q|<1.$
It follows from \cite{Andrews2019} that $\sigma mex(n)$ is almost always even and is odd exactly when $n$ is of the form $j(3j\pm 1).$
The idea of minimal excludant for the parts of a partition can be restricted to parts in a specific arithmetic progression. In the same article  \cite{Andrews2019},  Andrews and Newman defined  minimal odd excludant function, $moex(\pi)$, to be the smallest odd integer that is not a part of $\pi$.
In addition, for each positive integer $n$, they defined $$\sigma moex(n):=\sum_{\pi \in \mathcal{P}(n)} moex (\pi),$$ and  proved that the generating function for $\sigma moex(n)$ is given by 
\begin{align}\label{1.2}
\sum_{n=0}^{\infty}\sigma moex(n) q^n = (-q;q)_{\infty}(-q;q^2)_{\infty}^2.
\end{align}
 In this article, we study the arithmetic properties of Fourier coefficients of certain integer weight modular forms to  understand  the arithmetic properties of $\sigma mex(n)$ in a better way. We also present a parity result and distribution for  $\sigma moex(n)$.
Let us note that $\sigma mex(2n+1)$ is always even, for any positive integer $n$, as
\begin{align*}
 \sum_{n=0}^{\infty}\sigma mex(n) q^n = (-q;q)_{\infty}^2\equiv (q^2;q^2)_{\infty} \Mod{2}.
 \end{align*}
Recently, Ray and Barman  \cite[Theorem~1.6]{ray2020} studied the  divisibility of Uncu's partition function \cite{Uncu2018}, $\mathcal{EO}_u(n)$  by $2^k$, where $k$ is any positive integer. After some elementary calculations we observed that the generating function of $\mathcal{EO}_u(2n)$ and $\sigma mex(n)$ are the same. Therefore, for any positive integer $k$, we have
\begin{align*}
		\lim_{X\to\infty} \frac{\# \left\{0<n\leq X: \sigma mex(n)\equiv 0\Mod{2^k}\right\}}{X}&=1.
		\end{align*}
In fact, for almost every non-negative integer $N$ in an arithmetic progression $Tn+R$, 
the integer $\sigma mex(N)$ is a multiple of $2^k$. 
  On the other hand, there exists a positive number $\alpha$  such that there are at most $\mathcal{O}\left(\frac{X}{log^{\alpha}X}\right)$ many positive integer $n\leq X$ for which $\sigma mex(n)$ is not divisible  by $2^k$. 
Here we find  infinitely many arithmetic progressions where congruences for $\sigma mex(n)$ hold. Throughout this article, by $p \equiv a_1, a_2,\cdots, a_k  \Mod {M}$ we mean  $ p\equiv  i \Mod {M}$ for  $i\in \{a_1, a_2,\cdots, a_k \}$.\\

 \noindent  We deduce the following infinite families of congruence for $\sigma mex(n)$ modulo $4$ using the theory of Hecke eigenforms. 
 \begin{theorem}\label{thm2}
 	Let $k, n$ be non-negative integers. For each $i$ with $1\leq i \leq k+1$, consider the prime numbers $p_i$ such that $p_i \geq 5$ and  $p_i  \equiv 5, 7, 11 \Mod {12}$. Then, for any integer $j \not\equiv 0 \Mod {p_{k+1}}$, we have
 	\begin{align*} 
 	{\sigma mex}\left(p_1^2\dots p_{k+1}^2n + p_1^2\dots p_{k}^2p_{k+1}j +\frac{p_1^2\dots p_{k}^2p_{k+1}^2-1}{12}\right) \equiv 0 \Mod 4.
 	\end{align*}
 \end{theorem}
\noindent Now if we assume $p_1= p_2= \cdots=p_{k+1}=p$ in Theorem \ref{thm2}, then we derive
 \begin{align*} 
 \sigma mex\left(p^{2(k+1)}n + p^{2k+1}j +\frac{p^{2(k+1)}-1}{12}\right) \equiv 0 \Mod 4,
 \end{align*}
 where  $j \not\equiv 0 \Mod p$.
 In particular, for all $n\geq 0$ and $j\equiv 1, 2, 3, 4\Mod{5}$, we have
 \begin{align*} 
 \sigma mex\left(25n + 5j + 2\right) \equiv 0 \Mod{4}.
 \end{align*}

 Additionally, we prove the following multiplicative formulas for $\sigma mex(n)$ modulo~$4$.
 \begin{theorem} \label{thm2.2_2}
 	Let $k$ be a positive integer and $i \in \{5, 7, 11\}$. Suppose $p$ is a prime number such that $p \equiv i \Mod {12}$. Let $ \delta $ be a non-negative integer such that $p$ divides $12\delta  + i$, then
 	\begin{align*} 
 	{\sigma mex}\left(p^{k+1} n+ p\delta + \frac{pi-1}{12}\right) \equiv f(p)~{\sigma mex}\left(p^{k-1}n + \frac{12\delta +i-p}{12p}\right) \Mod {4},
 	\end{align*}
 	where $ f(p)$ is defined by$$ f(p)=
 	\begin{cases}
 	-1 & \text{if $p \equiv 5 \Mod {12}$,}\\
 	1 & \text{if $p \equiv 7, 11 \Mod {12}$.}\\
 	\end{cases}$$ 		
 \end{theorem}
 \begin{corollary}\label{cor2}
 	Let $k$ be a positive integer and $p$ be a prime number such that $p  \equiv 5, 7, 11 \Mod {12}$. Then 
 	\begin{align*} 	
 	{\sigma mex}\left(\frac{p^{2k}(12n+1)-1}{12}\right) \equiv  f(p)^k  ~~ {\sigma mex}(n) \Mod 4.
 	\end{align*}
 \end{corollary}
 \noindent In particular, if we take $p=5, 7, 11$ and $k=1$,  then from Corollary~\ref{cor2} we have the following:
 \begin{align*} 
 {\sigma mex}\left(n\right) \equiv (-1)\cdot{\sigma mex}\left(25n+ 2\right)  \equiv ~{\sigma mex}\left(49n+ 4\right)  \equiv ~{\sigma mex}\left(121n+ 10\right) \Mod{4} .
 \end{align*}
 \\
Next we study  $\sigma moex(n)$ and establish the following results.
\begin{theorem}\label{theorem222}
 	For every positive integer $n$, we have
 	\begin{align*}
 	\left\{1\leq n \leq X: \sigma moex(n)~ \text{is an even integer} \right\} \geq \alpha \log \log X,
 	\end{align*}
 	where $\alpha>0$ is a constant.
 \end{theorem}

 \begin{theorem}\label{theorem111}
 	For every positive integer $n$, we have
 	\begin{align*}
 	\left\{1\leq n \leq X: \sigma moex(n)~ \text{is an odd integer} \right\} \geq \beta \log \log X,
 	\end{align*}
 	where $\beta>0$ is a constant.
 \end{theorem}
We use  $Mathematica$ \cite{mathematica} for our computations.

\section{Proofs of Theorem~\ref{thm2} and  Theorem~\ref{thm2.2_2}}
In this  section, we prove Theorem~\ref{thm2} and  Theorem~\ref{thm2.2_2}. Before that  we recall some definitions and facts relating to the arithmetic of classical modular forms.
Throughout this chapter, we consider only integer weight modular forms and $N$ to be a positive integer. For  more details one can consult \cite{ koblitz1993, ono2004}.
\begin{definition}\cite[Definition 1.15]{ono2004}
	Let $\chi$ be a Dirichlet character modulo $N$. Then a modular form $f(z)\in M_{\ell}(\Gamma_1(N))$ has {\it Nebentypus character} $\chi$ if
	$$f\left( \frac{az+b}{cz+d}\right)=\chi(d)(cz+d)^{\ell}f(z)$$ for all $z\in \mathbb{H}$ and all $\begin{bmatrix}
	a  &  b \\
	c  &  d      
	\end{bmatrix} \in \Gamma_0(N)$. The space of such modular forms is denoted by $M_{\ell}(\Gamma_0(N), \chi)$.  Here $\Gamma_0(N):=\text{SL}_2 (\mathbb{Z})\cap \begin{bsmallmatrix}
	\mathbb{Z}  &  \mathbb{Z} \\
	N\mathbb{Z}  &  \mathbb{Z}      
	\end{bsmallmatrix}$  is the Hecke congruence subgroup of level $N$.
\end{definition}
\noindent Let $m$ be a positive integer and $f(z) = \displaystyle \sum_{n=0}^{\infty} a(n)q^n \in M_{\ell}(\Gamma_0(N),\chi)$. Then the action of the Hecke operator $T_m$ on $f(z)$ is defined by 
	\begin{align}
\label{hecke1}	f(z)|T_m := \sum_{n=0}^{\infty} \left(\sum_{d\mid \gcd(n,m)}\chi(d)d^{\ell-1}a\left(\frac{nm}{d^2}\right)\right)q^n.
	\end{align}
We note that $a(n)=0$ unless $n$ is a non-negative integer. The modular form $f(z)$ is called a {\it Hecke eigenform} if for every $m\geq2$ there exists a complex number $\lambda(m)$ for which 
	\begin{align}\label{hecke3}
	f(z)|T_m = \lambda(m)f(z).
	\end{align}
In  many  cases  the  generating  functions  for  combinatorial  objects are closely related to modular forms, particularly to  $\prod_{\delta\mid N}\eta(\delta z)^{r_\delta}$ (called an eta quotient), where $r_{\delta}$ is an integer and the Dedekind's eta function is defined by
\begin{align*}
\eta(z):=q^{1/24}(q;q)_{\infty}=q^{1/24}\prod_{n=1}^{\infty}(1-q^n)
\end{align*} 
with  $q=e^{2\pi iz}$ and $z\in \mathbb{H}$ (the complex upper half-plane). Gordon, Hughes, and Newman  \cite{Gordon1993, newman, newman1} established the following theorem regarding the modular properties of eta quotients.
\begin{theorem}\cite[Theorem 1.64]{ono2004}\label{thm_ono1} Suppose $f(z)=\displaystyle\prod_{\delta\mid N}\eta(\delta z)^{r_\delta}$ 
	is an eta quotient such that 
	\begin{align*}
		\ell=\displaystyle\frac{1}{2}\sum_{\delta\mid N}r_{\delta}\in \mathbb{Z},\quad  \quad \sum_{\delta\mid N} \delta r_{\delta}\equiv 0 \Mod{24}, \quad \text{and}  \quad \sum_{\delta\mid N} \frac{N}{\delta}r_{\delta}\equiv 0 \Mod{24}.
	\end{align*}	
	Then 
	$$
	f\left( \frac{az+b}{cz+d}\right)=\chi(d)(cz+d)^{\ell}f(z)
	$$
	for every  $\begin{bmatrix}
	a  &  b \\
	c  &  d      
	\end{bmatrix} \in \Gamma_0(N)$. Here the Nebentypus character
	$\chi(d)$ is given by $\left(\frac{(-1)^{k} \prod_{\delta\mid N}\delta^{r_{\delta}}}{d}\right).$
\end{theorem} 
\noindent Suppose that $f$ is an eta quotient satisfying the conditions of Theorem \ref{thm_ono1}.  Additionally, if $f$ is also holomorphic at all of the cusps of $\Gamma_0(N)$, then $f\in M_{\ell}(\Gamma_0(N), \chi)$.  To check the holomorphicity at cusps of $f(z)$ it suffices to check that the orders at the cusps are non-negative. The following theorem of Ligozat~\cite{Ligozat1972} gives the necessary criterion for determining orders of an eta quotient at a cusp. 
\begin{theorem}\cite[Theorem 1.65]{ono2004}\label{thm_ono1.1}
	Let $c, d,$ and $N$ be positive integers with $d\mid N$ and $\gcd(c, d)=1$. If $f(z)$ is an eta quotient satisfying the conditions of Theorem~\ref{thm_ono1} for $N$,then the order of vanishing of $f(z)$ at the cusp $\frac{c}{d}$ 
	is $$\frac{N}{24}\sum_{\delta\mid N}\frac{\gcd(d,\delta)^2r_{\delta}}{\gcd(d,\frac{N}{d})d\delta}.$$
\end{theorem}
In the next lemma we see that the generating  function for  $\sigma mex(n)$ is related to modular forms.
\begin{lemma} \label{lemma2.1}For any positive integer $n$, we have 
	\begin{align*}
		\sum_{n=0}^{\infty}\sigma mex(n)q^{12n+1}\equiv \eta(12z)^{2} \Mod {4},
	\end{align*}
where $\eta(12z)^2\in M_1(\Gamma_0(144),(\frac{-1}{d}))$ is an Hecke eigenform.	
\end{lemma}
 
\begin{proof}
	From \eqref{1.1} we have 
	\begin{align*}
	\sum_{n=0}^{\infty}\sigma mex(n)q^n=\frac{(q^2; q^2)_{\infty}^2}{(q; q)_{\infty}^2}\equiv(q; q)_{\infty}^2\Mod{4}.
	\end{align*}
This gives
\begin{align*}
	\sum_{n=0}^{\infty}{\sigma mex}(n)q^{12n+1} \equiv \eta(12z)^2\Mod 4.
\end{align*}
Using Theorem \ref{thm_ono1} and Theorem \ref{thm_ono1.1}, we have $\eta(12z)^2\in M_1(\Gamma_0(144),(\frac{-1}{d}))$. From \cite{martin}, we know that $\eta(12z)^2$ is a Hecke eigenform.
\end{proof}	
In the following lemma we prove some arithmetic properties of the Fourier coefficients of  $\eta(12z)^2$.
\begin{lemma}\label{lemma2.1.2}
	Suppose $\eta(12z)^2$ has a Fourier series expansion $\displaystyle\sum_{n=1}^{\infty} a(n)q^n$ and $p$ be a prime number such that $p \not\equiv 1 \Mod {12}$.  Then 
	\begin{align}\label{1.2.1}
	a(p^2n + pr) &= 0 ~~ \text{if} ~~ p\nmid n,\\
\label{1.3}	a(p^2n) + \left(\frac{-1}{p}\right) a(n)&= 0 ~~ \text{otherwise}.
	\end{align}	
\end{lemma}
\begin{proof}
	We have $\eta(12z)^2 = q-2q^{13}-q^{25}+2q^{37}+q^{49}+2q^{61} \dots=\displaystyle\sum_{n=1}^{\infty} a(n)q^n$. It is easy to observe that  $a(n) = 0$ if $n\not\equiv 1\Mod{12}$.  From Lemma \ref{lemma2.1} we know that  $\eta(12z)^2$ is a Hecke eigenform. Using \eqref{hecke1} and \eqref{hecke3} we get
		\begin{align*}
	\eta(12z)^2|T_p = \sum_{n=1}^{\infty} \left(a(pn) + \left(\frac{-1}{p}\right) a\left(\frac{n}{p}\right) \right)q^n = \lambda(p) \sum_{n=1}^{\infty} a(n)q^n.
	\end{align*}
	Equating the coefficients on both sides, we have the following
	\begin{align}\label{1.1.1}
	a(pn) + \left(\frac{-1}{p}\right) a\left(\frac{n}{p}\right) = \lambda(p)a(n).
	\end{align}
Since we consider the prime numbers $p \not\equiv 1 \Mod{12}$, putting $n=1$ in \eqref{1.1.1},  we obtain $a(p) = 0 = \lambda(p)$. Therefore from \eqref{1.1.1} we have 
	\begin{align}\label{new-1.1}
	a(pn) + \left(\frac{-1}{p}\right) a\left(\frac{n}{p}\right)= 0
	\end{align}
 for all prime $p \not\equiv 1 \Mod {12}.$ We conclude our proof from \eqref{new-1.1} by  taking $n=pn+r$  if $p\nmid n$, and  $n=pn$ otherwise.
	\end{proof}

\begin{proof}[Proof of Theorem \ref{thm2}] From Lemma \ref{lemma2.1} and Lemma \ref{lemma2.1.2}, we have
\begin{align*}
	\sum_{n=0}^{\infty}{\sigma mex}(n)q^{12n+1} &\equiv \eta(12z)^2\Mod 4\\
	&=\sum_{n=1}^{\infty} a(n)q^n.
\end{align*}

Therefore
\begin{align}\label{main}
\sigma mex(n)\equiv a(12n+1) \Mod{4}.
\end{align}
Substituting $n$ by $12n-pr+1$ in \eqref{1.2.1} and then applying the congruence relation  \eqref{main}, we have
	\begin{align}\label{1.4}
	\sigma mex\left(p^2n + \frac{pr(1-p^2)}{12} + \frac{p^2-1}{12}\right) \equiv 0\Mod {4}.
	\end{align}	
	Similarly, substituting $n$ by $12n+1$ in \eqref{1.3} and using the congruence relation \eqref{main}, we obtain
	\begin{align}\label{1.5}
	\sigma mex\left(p^2n + \frac{p^2-1}{12} \right)\equiv (-1) \left(\frac{-1}{p}\right) \sigma mex(n)\Mod {4}.
	\end{align}	
	Since $p \geq 5$ is prime such that $p \equiv 5,7, 11 \Mod{12}$, then $12\mid (p^2-1)$ and $\gcd \left(\frac{p^2-1}{12} , p\right) = 1$.  
	Hence if $r$ runs over a residue system excluding the multiple of $p$, so does $\frac{1-p^2}{12}r$.
    Thus  we can rewrite \eqref{1.4} as
	\begin{align}\label{1.6}
	\sigma mex\left(p^2n + \frac{p^2-1}{12} + pj\right) \equiv 0\Mod {4},
	\end{align}
where $j$ is not a multiple of $p$.	Then, for the prime number  $p_i \geq 5$  such that $p_i \equiv 5,7, 11 \Mod{12}$, we get
	\begin{align*}
	 &\sigma mex\left(p_1^2\dots p_{k}^2n + \frac{p_1^2\dots p_{k}^2-1}{12}\right)\\
	=&\sigma mex\left(p_1^2\left(p_2^2\dots p_{k}^2n + \frac{p_2^2\dots p_{k}^2-1}{12}\right)+\frac{p_1^2-1}{12}\right)\\
	\equiv&(-1) \left(\frac{-1}{p_1}\right) \sigma mex\left(p_2^2\dots p_{k}^2n + \frac{p_2^2\dots p_{k}^2-1}{12}\right)\Mod {4}.
	\end{align*}
The above congruence follows from \eqref{1.5}.  By repeating the   process $(k-1)$~times, we obtain 
		\begin{align}\label{1.7}
\notag &\sigma mex\left(p_1^2\dots p_{k}^2n + \frac{p_1^2\dots p_{k}^2-1}{12}\right)\\
&\equiv (-1)^k \left(\frac{-1}{p_1}\right)\left(\frac{-1}{p_2}\right)\dots \left(\frac{-1}{p_k}\right)\sigma mex(n)\Mod {4}.
	\end{align}
	Let $j\not\equiv 0\Mod{p_{k+1}}$. Then, \eqref{1.6} and \eqref{1.7} yield
	\begin{align*}
	&{\sigma mex}\left(p_1^2\dots p_{k+1}^2n + p_1^2\dots p_{k}^2p_{k+1}j +\frac{p_1^2\dots p_{k}^2p_{k+1}^2-1}{12}\right)\\
&\equiv (-1)^k \left(\frac{-1}{p_1}\right)\left(\frac{-1}{p_2}\right)\dots \left(\frac{-1}{p_k}\right)\sigma mex\left(p^2n + \frac{p^2-1}{12} + pj\right)\Mod {4}\\
& \equiv 0 \Mod 4.
	\end{align*}
	This concludes the theorem.
\end{proof}
\noindent The next result reveals the multiplicative nature of  $\sigma mex(n)$ and we get Corollary~\ref{cor2} as a special case of Theorem~\ref{thm2.2_2}.
\begin{proof}[Proof of Theorem~\ref{thm2.2_2}]	
	From \eqref{new-1.1}, we have 
	\begin{align}\label{thm2.2.2}
	a(pn) = (-1)\left(\frac{-1}{p}\right) a\left(\frac{n}{p}\right),
	\end{align}
where the prime $p \not\equiv 1\Mod {12}$.	 Let $i$ be a fixed number and  $ i \in \{5,7,11\}$. Suppose  $\delta$ is a non-negative integer such that  the prime $p\equiv i \Mod {12}$ divides $12\delta  + i$. Then, substituting $n$ by $12p^kn+12\delta+i$ in \eqref{thm2.2.2}, we get the following 
\begin{align}\label{thm2.2.4}
a\left(12\left( p^{k+1}n+ p\delta +\frac{pi-1}{12}\right)+1\right)= (-1)\left(\frac{-1}{p}\right) a\left(12\left(p^{k-1}n+ \frac{12\delta +i-p}{12p}\right)+1\right).
\end{align}	
Note that  $\frac{pi-1}{12}$ and $\frac{12\delta+i-p}{12p}$  are integers. Now, using \eqref{main} and \eqref{thm2.2.4} we obtain
	
	\begin{align} \label{thm2.2.4.11}
	{\sigma mex}\left(p^{k+1} n+ p\delta + \frac{pi-1}{12}\right) \equiv(-1)\left(\frac{-1}{p}\right){\sigma mex}\left(p^{k-1}n + \frac{12\delta +i-p}{12p}\right)  \Mod {4}.
	\end{align}
	For a prime $p\geq 5$ and $p\not\equiv1\Mod{12}$, we know that $$\left(\frac{-1}{p}\right)=
	\begin{cases}
	1 & \text{if $p \equiv 5 \Mod {12}$,}\\
	-1 & \text{if $p \equiv 7, 11 \Mod {12}$.}\\
	\end{cases}$$ 	
	Hence  we conclude the proof of the theorem from \eqref{thm2.2.4.11}.
\end{proof}	
\begin{proof}[Proof of Corollary~\ref{cor2}] Let $ i \in \{5,7,11\}$. For a fixed $i$, we consider a prime $p\equiv i\Mod {12}$  and a non-negative integer $\delta$  such that  the prime $p$ divides $12\delta  + i$ as in Theorem~\ref{thm2.2_2}. Then by replacing $n$ with $np^{k-1}$ in \eqref{thm2.2.4.11} we have
		\begin{align} \label{thm2.2.4.11.1.1}
	{\sigma mex}\left(p^{2k} n+ p\delta + \frac{pi-1}{12}\right) \equiv(-1)\left(\frac{-1}{p}\right){\sigma mex}\left(p^{2(k-1)}n + \frac{12\delta +i-p}{12p}\right)  \Mod {4}.
	\end{align}
If we consider those cases when $12\delta+i=p^{2k-1}$ for $k\geq1$, then from \eqref{thm2.2.4.11.1.1} we obtain
	\begin{align*}
	{\sigma mex}\left(\frac{p^{2k}(12n+1)-1}{12}\right) \equiv (-1)\left(\frac{-1}{p}\right) {\sigma mex}\left(\frac{p^{2(k-1)}(12n+1)-1}{12}\right) \Mod 4.
	\end{align*}
By using the above recursive relation $k-1$~times, we obtain 	
	\begin{align*}
	{\sigma mex}\left(\frac{p^{2k}(12n+1)-1}{12}\right) \equiv \left((-1)\left(\frac{-1}{p}\right)\right)^k {\sigma mex}(n) \Mod 4.
	\end{align*}		
	Then the result directly follows from the above congruence.
\end{proof}

\section{Proof of Theorem~\ref{theorem222} and \ref{theorem111}} 
 To prove Theorem ~\ref{theorem222}--\ref{theorem111}  we need the following lemmas. We prove Lemma \ref{lemma222}--\ref{lemma111} inspired by  Kolberg~\cite{Kolberg}.
 \begin{lemma}\label{lemma1}
 	For any positive integer $n$, we have 
 	\begin{align*}
 \sum_{k=0}^{\infty}\sigma moex\left(n-\frac{k(3k-1)}{2}\right) +\sum_{k=1}^{\infty}\sigma moex\left(n-\frac{k(3k+1)}{2}\right) \equiv 0 \Mod{2}.
 	\end{align*}
 \end{lemma}
 	\begin{proof}
   For $|ab|<1$, Ramanujan's general theta function $f(a, b)$ is defined as
 	\begin{align*}
 	f(a, b)=\sum_{n=-\infty}^{\infty}a^{n(n+1)/2}b^{n(n-1)/2}.
 	\end{align*}
 	In Ramanujan's notation, the Jacobi triple product identity \cite[Entry 19, p. 36]{Berndt1991} takes the shape 
 	\begin{align*}
 	\index{$f(a,b)$}	f(a,b)=(-a;ab)_{\infty}(-b;ab)_{\infty}(ab;ab)_{\infty}.
 	\end{align*}
One of the most important special cases of $f(a,b)$ is
 	\begin{align*}
f(-q, -q^2)=(q;q)_{\infty}=\sum_{n=-\infty}^{\infty}(-1)^nq^{n(3n-1)/2}.
 	\end{align*}
 	Therefore, we have
 	\begin{align}\label{2.1}
 	(q;q)_{\infty}\equiv \sum_{n=0}^{\infty}q^{n(3n-1)/2}+\sum_{n=1}^{\infty}q^{n(3n+1)/2} \Mod {2}.
 	\end{align}
 	From \eqref{1.2},  the generating function for $\sigma moex(n)$ is 
 	\begin{align*}
 	\sum_{n=0}^{\infty}\sigma moex(n) q^n= (-q;q)_{\infty}(-q;q^2)_{\infty}^2.
 	\end{align*}
 	Using the binomial theorem and \eqref{2.1}, we have 
 	\begin{align*}
 	\sum_{n=0}^{\infty}\sigma moex(n) q^n \equiv  \frac{1}{(q;q)_{\infty}} \equiv \frac{1}{ \displaystyle \sum_{n=0}^{\infty}q^{n(3n-1)/2}+\sum_{n=1}^{\infty}q^{n(3n+1)/2}} \Mod{2}.
 	\end{align*}
 	Therefore,  
 	\begin{align*}
 	\sum_{n=0}^{\infty}\left(\sum_{k=0}^{\infty}\sigma moex\left(n-\frac{k(3k-1)}{2}\right) +\sum_{k=1}^{\infty}\sigma moex\left(n-\frac{k(3k+1)}{2}\right)\right)q^n \equiv 1 \Mod{2}.
 	\end{align*}
 	Hence, for any positive integer $n$, comparing the coefficients of $q^n$ on both sides of the above congruence, we get the result.
 	\end{proof}
 \begin{lemma}\label{lemma222}
 	For every positive integer $\ell\geq 2$, there exists an integer $n\in \left[\ell, \frac{\ell(3\ell+1)}{2}\right]$ such that $\sigma moex(n)$ is an even integer. In particular, there are infinitely many integers $n$ for which  $\sigma moex(n)$ is an even integer.
 \end{lemma}	
\begin{proof} We use proof by contradiction. For every positive integer $\ell\geq 2$, we assume that $\sigma moex(n)$ is odd  when $n\in \left[\ell, \frac{\ell(3\ell+1)}{2}\right]$. We define $\mathcal{C}(k):= \frac{\ell(3\ell+1)}{2}-\frac{k(3k-1)}{2}$  and $\mathcal{D}(k):= \frac{\ell(3\ell+1)}{2}-\frac{k(3k+1)}{2}$ for a positive integer $k$.  From Lemma~\ref{lemma1}, we have 
	\begin{align}\label{2.4}
	\sum_{k=0}^{\infty}\sigma moex\left(\mathcal{C}(k)\right) +\sum_{k=1}^{\infty}\sigma moex\left(\mathcal{D}(k)\right) \equiv 0 \Mod{2}.
	\end{align} 	
It is easy to check that $\mathcal{C}(k)$ and $\mathcal{D}(k)$ are negative for all $k\geq \ell+1$.  Since $\sigma moex(n)=0$ for all negative integers $n$, then from \eqref{2.4} we have
	\begin{align} \label{new2.2.1}
	&\sum_{k=0}^{\ell}\sigma moex\left(\mathcal{C}(k)\right) +\sum_{k=1}^{\ell}\sigma moex\left(\mathcal{D}(k)\right) \\
\notag	=&\sum_{k=0}^{\ell}\sigma moex\left(\mathcal{C}(k)\right) +\sum_{k=1}^{\ell-1}\sigma moex\left(\mathcal{D}(k)\right) + \sigma moex\left(\mathcal{D}(\ell)\right).
	\end{align}
From \eqref{1.2} we have 
\begin{align*}
\sum_{n=0}^{\infty}\sigma moex(n) q^n =1 + 3 q + 4 q^2 + 7 q^3 + 13 q^4 + 19 q^5 + 29 q^6+\cdots. 
\end{align*}
Therefore $\sigma moex\left(\mathcal{D}(\ell)\right)=\sigma moex\left(0\right)=1$. For $\ell\geq 2$, we have $\mathcal{D}(\ell-m)=\frac{1}{2}\big(6\ell m-3m^2+m\big).$ Then, for $m = 1, 2, \cdots, \ell-1$ we see that
\begin{align*}
   \mathcal{D}(\ell-m)\geq\frac{1}{2}\big(6\ell -3\ell+4\big)\geq \ell.
\end{align*}
So, $\mathcal{D}(k) \in \left[\ell, \frac{\ell(3\ell+1)}{2}\right]$ for $k \in \{1,2, \dots, \ell-1\}$. In a similar manner for $k \in \{0,1,2, \dots, \ell\}$ we can show that $\mathcal{C}(k) \in \left[\ell, \frac{\ell(3\ell+1)}{2}\right]$. 
By our assumption 
	$\displaystyle\sum_{k=0}^{\ell}\sigma moex\left(\mathcal{C}(k)\right)+\displaystyle\sum_{k=1}^{\ell-1}\sigma moex\left(\mathcal{D}(k)\right)$ is even since  it is a sum of  $(\ell+1)+(\ell-1)=2\ell$ odd numbers. Consequently, the summation  \eqref{new2.2.1} is  odd for $\ell\geq 2$, which is a contradiction to the fact \eqref{2.4}. This concludes our result.
\end{proof}
\begin{proof}[Proof of Theorem~\ref{theorem222}] Suppose $n$ is a positive number. Now we count the number of elements of the set  $$\left\{1\leq n \leq X: \sigma moex(n)~ \text{is an even integer} \right\}.$$
We consider $a_1=2$, $a_{k}= \dfrac{a_{k-1}(3a_{k-1}+1)}{2}$ and a partition of the interval $[1,X]$ as follows
$$[1,X]=[1,a_1)\cup[a_1,a_2)\cup\dots\cup[a_{k-1},a_{k})\cup \dots \cup[a_{\nu},X],$$ 
where $\nu$  is the largest integer such that $a_{\nu}\leq X.$ By Lemma \ref{lemma222} we can find a positive number $n$ in $\left[a_{k-1}, a_{k}\right]$ such that $\sigma moex(n)$ is an even integer. Then the number of $n\leq X$  for which $\sigma moex(n)$ is even  is at
least $\lfloor\nu/2\rfloor$. It remains to find the value of $\nu$ as a function of $X$. For all $k\geq 0$, we have
\begin{align*}
 a_{k} = \frac{a_{k-1}(3a_{k-1}+1)}{2} \leq 2 a_{k-1}^2 \leq 2^{2^{k-1}-1} a_1^{2^{k-1}}\leq 2^{2^k}.
\end{align*}
Since $a_{\nu}\leq X< a_{\nu+1}$,  we see that  $\nu\geq \alpha \log\log X$ for some  constant  $ \alpha>0 $.
\end{proof}
 \begin{lemma}\label{lemma111}
 For every positive integer $\ell\geq2$, there exists an integer $n\in \left[2\ell-1, \frac{\ell(3\ell-1)}{2}\right]$ such that $\sigma moex(n)$ is an odd integer. Consequently, there are infinitely many integers $n$ for which  $\sigma moex(n)$ is an odd integer.
 \end{lemma}
\begin{proof} We use the method of contradiction to prove the result.  For every positive integer $\ell\geq 2$, assume that $\sigma moex(n)$ is an even number when $n\in \left[2\ell-1, \frac{\ell(3\ell-1)}{2}\right]$. For a positive integer $k$, let us consider $\mathcal{A}(k):= \frac{\ell(3\ell-1)}{2}-\frac{k(3k-1)}{2}$  and $\mathcal{B}(k):= \frac{\ell(3\ell-1)}{2}-\frac{k(3k+1)}{2}$.  From Lemma~\ref{lemma1} we have 
	\begin{align}\label{2.2}
	\sum_{k=0}^{\infty}\sigma moex\left(\mathcal{A}(k)\right) +\sum_{k=1}^{\infty}\sigma moex\left(\mathcal{B}(k)\right) \equiv 0 \Mod{2}.
	\end{align} 	
It is easy to check that, for all $k\geq \ell+1$ the value of $\mathcal{A}(k)<0$, and  $\mathcal{B}(k)<0$ for all $k\geq \ell$. Since $\sigma moex(n)=0$ for all negative integers $n$, it follows from \eqref{2.2} that
	\begin{align} \label{2.3}
	\notag	&\sum_{k=0}^{\ell}\sigma moex\left(\mathcal{A}(k)\right) +\sum_{k=1}^{\ell-1}\sigma moex\left(\mathcal{B}(k)\right) \\
	\notag=&\sum_{k=0}^{\ell-1}\sigma moex\left(\mathcal{A}(k)\right) +\sum_{k=1}^{\ell-1}\sigma moex\left(\mathcal{B}(k)\right) + \sigma moex\left(\mathcal{A}(\ell)\right) \\
	=&\sum_{k=0}^{\ell-1}\sigma moex\left(\mathcal{A}(k)\right) +\sum_{k=1}^{\ell-1}\sigma moex\left(\mathcal{B}(k)\right) +1.
	\end{align}
Note that, for a fixed $k$, $\mathcal{A}(k)$ and $\mathcal{B}(k)$ are decreasing functions of $k$. Then for every positive integer $\ell\geq2$, we check that  $\mathcal{A}(0),~\mathcal{A}(\ell-1), ~\mathcal{B}(1)~\text{and}~\mathcal{B}(\ell-1) \in \left[2\ell-1, \frac{\ell(3\ell-1)}{2}\right].$
Thus, for $\ell\geq2$ and $k\in \{1,2, \dots, \ell-1\}$ the values of  $\mathcal{A}(k),~\mathcal{B}(k)$, and $\mathcal{A}(0)$ lies in the interval $\left[2\ell-1, \frac{\ell(3\ell-1)}{2}\right]$. By our assumption
	$\displaystyle\sum_{k=0}^{\ell-1}\sigma moex\left(\mathcal{A}(k)\right)$ and $\displaystyle\sum_{k=1}^{\ell-1}\sigma moex\left(\mathcal{B}(k)\right) $ are even numbers. Therefore the summation \eqref{2.3} 
	is odd for $\ell\geq 2$, which is a contradiction to the fact \eqref{2.2}. This concludes our result.
\end{proof}

\begin{proof}[Proof of Theorem ~\ref{theorem111}] Suppose $n$ is a positive number. Now we count the number of elements of the set  $$\left\{1\leq n \leq X: \sigma moex(n)~ \text{is an odd integer} \right\}.$$
Consider $a_1=5$ and  $a_{k}= \dfrac{a_{k-1}(3a_{k-1}-1)}{2}$ for $k\geq 2$. Let us take a partition the interval $[1,X]$ as follows
$$[1,X]=[1,5)\cup[5,a_2)\cup\dots\cup[a_{k-1},a_{k})\cup \dots \cup[a_{\nu},X],$$ 
where $\nu$  is the largest integer such that $a_{\nu}\leq X.$ From Lemma \ref{lemma111},  we can find a positive number $n$ in $\left[a_{k-1}, a_{k}\right]$ such that $\sigma moex(n)$ is an odd integer. Then the number of $n\leq X$  for which $\sigma moex(n)$ is odd  is at
least $\lfloor\nu/2\rfloor$, and it remains to find the value of $\nu$ as a function of $X$. Now for all $k\geq 0$, we get,
\begin{align*}
 a_{k} = \frac{a_{k-1}(3a_{k-1} -1)}{2} \leq \frac{3}{2} a_{k-1}^2 \leq\left( \frac{3}{2}\right)^{2^{k-1}-1}a_1^{2^{k-1}}=\left( \frac{3}{2}\right)^{2^{k-1}-1}5^{2^{k-1}}\leq5^{2^k}
\end{align*}
and hence we get the result immediately.
\end{proof}

\subsection*{Acknowledgements} 
The author is  grateful to Prof.\ Bruce C.\ Berndt  for his valuable suggestions.  The author would like to thank the National Board for Higher Mathematics (Department of Atomic Energy, India) for the post-doctoral fellowship and the Indian Statistical Institute Delhi for providing the research facilities. The author would also like to thank the anonymous referees for the careful review and comments.
\bibliographystyle{siam}
\bibliography{mex}
\end{document}